\documentclass[12pt,leqno]{article}
\usepackage[doc]{optional}
\usepackage{color}
\usepackage{float}
\usepackage{soul}
\usepackage{graphicx}
\definecolor{labelkey}{rgb}{0,0.08,0.45}
\definecolor{refkey}{rgb}{0,0.6,0.0}
\definecolor{Brown}{rgb}{0.45,0.0,0.05}
\definecolor{lime}{rgb}{0.00,0.8,0.0}
\definecolor{lblue}{rgb}{0.5,0.5,0.99}
\usepackage{mathpazo}
\usepackage{amsmath}
\usepackage{amssymb}
\usepackage{theorem}
\newcommand{\nnn}{\ensuremath{{n\in{\mathbb N}}}}

\newcommand{\menge}[2]{\big\{{#1}~\big |~{#2}\big\}}
\newcommand{\mmenge}[2]{\bigg\{{#1}~\bigg |~{#2}\bigg\}}

\newcommand{\To}{\ensuremath{\rightrightarrows}}

\newcommand{\fenv}[1]%
{\ensuremath{\,\overrightarrow{\operatorname{env}}_{#1}}}
\newcommand{\benv}[1]%
{\ensuremath{\,\overleftarrow{\operatorname{env}}_{#1}}}
\newcommand{\emp}{\ensuremath{\varnothing}}

\newcommand{\scal}[2]{\left\langle{#1},{#2}  \right\rangle}

\newcommand{\zeroun}{\ensuremath{\left]0,1\right[}}
\newcommand{\RR}{\ensuremath{\mathbb R}}

\newcommand{\dom}{\ensuremath{\operatorname{dom}}}

\newcommand{\gr}{\ensuremath{\operatorname{gr}}}

\newcommand{\ran}{\ensuremath{\operatorname{ran}}}

\newcommand{\Fix}{\ensuremath{\operatorname{Fix}}}
\newcommand{\Id}{\ensuremath{\operatorname{Id}}}

\newcommand{\pinf}{\ensuremath{+\infty}}
\newcommand{\bx}{\ensuremath{\mathbf{x}}}
\newcommand{\bX}{\ensuremath{{\mathbf{X}}}}
\newcommand{\bR}{\ensuremath{{\mathbf{R}}}}
\newcommand{\bT}{\ensuremath{{\mathbf{T}}}}
\newcommand{\bN}{\ensuremath{{\mathbf{N}}}}
\newcommand{\bL}{\ensuremath{{\mathbf{L}}}}
\newcommand{\bJ}{\ensuremath{{\mathbf{J}}}}
\newcommand{\bD}{\ensuremath{{\boldsymbol{\Delta}}}}
\newcommand{\bS}{\ensuremath{{\mathbf{S}}}}
\newcommand{\bA}{\ensuremath{{\mathbf{A}}}}
\newcommand{\bB}{\ensuremath{{\mathbf{B}}}}

\newcommand{\by}{\ensuremath{\mathbf{y}}}

{\begin{list}{}{%
\settowidth{\labelwidth}{\textrm{#1~}}%
\setlength{\leftmargin}{\labelwidth+\labelsep}}}%requires macro calc.sty
{\end{list}}
\newtheorem{theorem}{Theorem}[section]
\newtheorem{lemma}[theorem]{Lemma}
\newtheorem{corollary}[theorem]{Corollary}
\newtheorem{proposition}[theorem]{Proposition}
\newtheorem{definition}[theorem]{Definition}
\theoremstyle{plain}{\theorembodyfont{\rmfamily}
}
\theoremstyle{plain}{\theorembodyfont{\rmfamily}
}
\theoremstyle{plain}{\theorembodyfont{\rmfamily}
}
\theoremstyle{plain}{\theorembodyfont{\rmfamily}
}
\newtheorem{fact}[theorem]{Fact}
\theoremstyle{plain}{\theorembodyfont{\rmfamily}
\newtheorem{remark}[theorem]{Remark}}

\newcommand{\boxedeqn}[1]{%
    \[\fbox{%
        \addtolength{\linewidth}{-2\fboxsep}%
        \addtolength{\linewidth}{-2\fboxrule}%
        \begin{minipage}{\linewidth}%
        \begin{equation}#1\\[+4mm]\end{equation}%
        \end{minipage}%
      }\]%
  }
  
\newcommand{\algJA}{\text{\texttt{alg}}(J_A)}
\newcommand{\algJR}{{\text{\texttt{alg}}(\bJ\circ\bR)}}
\newcommand{\algT}{\text{\texttt{alg}}(\bT)}

\allowdisplaybreaks % or locally if problems {\allowdisplaybreaks
\begin{document}

\title{\textrm{
Fixed Points of Averages
of Resolvents:\\
Geometry and Algorithms}}

\author{
Heinz H.\ Bauschke\thanks{Department of
Mathematics \& Statistics, University
of British Columbia,
Kelowna, B.C.\ V1V~1V7, Canada. E-mail:
\texttt{heinz.bauschke@ubc.ca}.},
Xianfu\ Wang\thanks{Department of Mathematics \& Statistics,
University of British Columbia,
Kelowna, B.C.\ V1V~1V7, Canada.
E-mail:  \texttt{shawn.wang@ubc.ca}.},
~and Calvin J.S.\ Wylie\thanks{Department of
Mathematics \& Statistics, University of
British Columbia,
Kelowna, B.C.\ V1V~1V7, Canada. E-mail:
\texttt{calvin.wylie@gmail.com}.}}

\date{February 7, 2011}
\maketitle

\vskip 8mm

\begin{abstract} \noindent
To provide generalized solutions if a given problem admits no actual
solution is an important task in mathematics and the natural sciences.
It has a rich history dating back to the early 19th century when Carl
Friedrich Gauss developed the method of least squares of a system of
linear equations --- its solutions can be viewed as fixed
points of averaged projections onto hyperplanes.
A powerful generalization of this problem is to find fixed
points of averaged resolvents (i.e., firmly nonexpansive mappings).

This paper concerns the relationship between the set of fixed points
of averaged resolvents and certain fixed point sets of compositions
of resolvents. It partially extends recent work for two
mappings on a question of C.\ Byrne. The analysis suggests a
reformulation in a product space.

Furthermore, two new algorithms are presented.
A complete convergence proof that is based on averaged mappings is
provided for the first algorithm. The second algorithm, which
currently has no convergence proof,
iterates a mapping that is not even nonexpansive.
Numerical experiments indicate the potential of these algorithms when
compared to iterating the average of the resolvents.
\end{abstract}

{\small
\noindent
{\bfseries 2010 Mathematics Subject Classification:}
{Primary 47H05, 47H09;
Secondary 47J25, 65K05, 65K10, 90C25.
% 26B25 convexity generalizations? no
% 47H10 fixed point theorems? no
}


}

\section{Introduction}

Throughout this paper,
\boxedeqn{
\text{$X$ is
a real Hilbert space with inner
product $\scal{\cdot}{\cdot}$ }
}
and induced norm $\|\cdot\|$. We impose that $X\neq\{0\}$.
To motivate the results of this paper,
let us assume that $C_1,\ldots,C_m$ are finitely many nonempty closed
convex subsets of $X$, with projections (nearest point mappings)
$P_1,\ldots,P_m$.
Many problems in mathematics and the physical sciences can be recast
as the \textbf{convex feasibility problem} of finding a point in the
intersection $C_1\cap\cdots\cap C_m$.
However, in applications it may well be that this intersection
is empty. In this case, a powerful and very useful
generalization of the intersection
is the set of fixed points of the operator
\begin{equation}
\label{e:mproj}
X\to X\colon x\mapsto \frac{P_1x+\cdots +P_mx}{m}.
\end{equation}
(See, e.g., \cite{Comb94} for applications.)
Indeed, these fixed points are precisely the minimizers of the
convex function
\begin{equation}
\label{e:minis}
X\to\RR\colon x\mapsto \sum_{i=1}^{m}\|x-P_ix\|^2
\end{equation}
and---when each $C_i$ is a suitably described hyperplane---there
is a well known connection to the set of
\textbf{least squares solutions} in the sense of linear algebra (see
Appendix~A).

A problem open for a long time is to find precise relationships
between the fixed points of the operator defined in \eqref{e:mproj}
and the fixed points of the composition $P_m\circ\cdots \circ P_2\circ P_1$
when the intersection $C_1\cap\cdots\cap C_m$ is empty.
(It is well known that both fixed points sets coincide with
$C_1\cap\cdots\cap C_m$ provided this intersection is nonempty.)
This problem was recently explicitly stated and nicely discussed
in \cite[Chapter~50]{Byrne05} and \cite[Open~Question~2 on page~101 in
Subsection~8.3.2]{Byrne08}.
For other related work\footnote{In passing, we mention that when
$C_1$, $C_2$, $C_3$ are line segments forming a triangle in the
Euclidean plane, then the minimizer of \eqref{e:minis} is known
as the \textbf{symmedian point} (also known as the
Grebe-Lemoine point) of the given triangle;
see \cite[Theorem~349 on page~216]{Johnson}.}, see
\cite{BBL}, \cite{BauEd}, \cite{CEG}, \cite{DeP}, and the references
therein.
When $m=2$, the recent work \cite{WangBau} contains some
precise relationships. For instance,
the results in \cite[Section~3]{WangBau} show that
\begin{equation}
\Fix(P_2\circ P_1) \to
\Fix\big(\tfrac{1}{2}P_1+\tfrac{1}{2}P_2\big)
\colon
x\mapsto \tfrac{1}{2}x + \tfrac{1}{2}P_1x
\end{equation}
is a well defined bijection.

{\em
Our goal in this paper is two-fold.
First, we wish to find a suitable extension to describe
these fixed point sets when $m\geq 3$.
Second, we build on these insights to obtain algorithms
for finding these fixed points.
}

The results provided are somewhat surprising. While we completely
generalize some of the two-set work from \cite{WangBau}, the generalized
intersection is \emph{not} formulated as the fixed point set of a
simple composition, but rather as the fixed point set of a more
complicated operator described in a product space.
Nonetheless, the geometric insight obtained will turn out to be
quite useful in the design
of new algorithms that show better convergence properties when
compared to straight iteration of the averaged projection operator.
Furthermore, the results actually hold for very general firmly
nonexpansive operators---equivalently, resolvents of maximally
monotone operators---although the optimization-based interpretation
as a set of minimizers analogous to \eqref{e:minis} is then
unavailable. 

The paper is organized as follows.
In the remainder of this introductory section, we describe some
central notions fundamental to our analysis.
The main result of Section~\ref{s:2} is Theorem~\ref{t:SvsFix}
where we provide a precise correspondence between the fixed point set
of an averged resolvent $J_A$ and a certain set $\bS$ in a product space.
In Section~\ref{s:3}, it is shown that $\bS$ is in fact the fixed
point set of an averaged mapping (see Corollary~\ref{c:punch}).
This insight is brought to good use in Section~\ref{s:2new}, where
we design a new algorithm for finding a point in $\bS$ (and hence
in $\Fix J_A$) and where we provide a rigorous convergence proof.
Akin to the Gauss-Seidel variant of the Jacobi iteration in numerical
linear algebra, we propose another new algorithm.
Numerical experiments illustrate that
this heuristic  algorithm performs very well; however, it still
lacks a rigorous proof of convergence. An appendix concludes the paper.
The first part of the
appendix connects fixed points of averages of projections
onto hyperplanes to classical least squares solutions, while
the second part contains some more technical observations
regarding the heuristic method.
The notation we utilize is standard and as in
\cite{BC2011}, \cite{BorVanBook},
\cite{Rock70}, \cite{Rock98}, \cite{Simons1},
\cite{Simons2}, or \cite{Zalinescu} to which we also refer for
background.

Recall that a mapping
\begin{equation}
T\colon X\to X
\end{equation}
is \textbf{firmly nonexpansive}
(see \cite{Zara} for the first systematic study)
if
\begin{equation}
(\forall x\in X)(\forall y\in X)
\quad
\|Tx-Ty\|^2 + \|(\Id-T)x-(\Id-T)y\|^2 \leq \|x-y\|^2,
\end{equation}
where $\Id\colon X\to X\colon x\mapsto x$ denotes
the \textbf{identity operator}.
The prime example of firmly nonexpansive mappings
are \textbf{projection operators} (also known
as nearest point mappings) with respect to nonemtpy closed convex
subsets of $X$.
It is clear that if $T$ is firmly nonexpansive,
then it is \textbf{nonexpansive}, i.e.,
{Lipschitz continuous} with constant $1$,
\begin{equation}
(\forall x\in X)(\forall y\in X) \quad
\|Tx-Ty\|\leq\|x-y\|;
\end{equation}
the converse, however, is false (consider $-\Id$).
%When $T$ is Lipschitz continuous with a constant in
%$\left[0,1\right[$, then we shall refer to $T$ as a
%\textbf{Banach contraction}.
The set of \textbf{fixed points} of $T$ is
\begin{equation}
\Fix T = \menge{x\in X}{x=Tx}.
\end{equation}

The following characterization of firm nonexpansiveness is well known
and will be used repeatedly.

\begin{fact}
\label{f:firm}
{\rm (See, e.g., \cite{BC2011,GK,GR}.)}
Let $T\colon X\to X$.
Then the following are equivalent:
\begin{enumerate}
\item $T$ is firmly nonexpansive.
\item $\Id-T$ is firmly nonexpansive.
\item $2T-\Id$ is nonexpansive.
\item $(\forall x\in X)(\forall y\in X)$
$\|Tx-Ty\|^2 \leq \scal{x-y}{Tx-Ty}$.
\item
\label{f:firmsymm}
$(\forall x\in X)(\forall y\in X)$
$0\leq \scal{Tx-Ty}{(\Id-T)x-(\Id-T)y}$.
\end{enumerate}
\end{fact}

Firmly nonexpansive mappings are also intimately tied
with maximally monotone operators.
Recall that a set-valued operator $A\colon X\To X$ (i.e., $(\forall x\in
X)$ $Ax\subseteq X$) with {graph} $\gr A$ is \textbf{monotone} if
\begin{equation}
(\forall (x,u)\in\gr A)
(\forall (y,v)\in\gr A)
\quad
\scal{x-y}{u-v}\geq 0,
\end{equation}
and that $A$ is \textbf{maximally monotone} if it is monotone and every proper
extension of $A$ fails to be monotone.
We write $\dom A =\menge{x\in X}{Ax\neq\varnothing}$
and $\ran A = A(X) = \bigcup_{x\in X}Ax$ for the
\textbf{domain} and \textbf{range} of $A$, respectively.
The \textbf{inverse} of $A$ is defined via
$\gr A^{-1} = \menge{(u,x)\in X\times X}{u\in Ax}$.
Monotone operators are ubiquitous in modern analysis and
optimization; see, e.g., the books
\cite{BC2011},
\cite{BorVanBook},
\cite{Brezis},
\cite{BurIus},
\cite{Simons1},
\cite{Simons2},
\cite{Zalinescu},
\cite{Zeidler2a},
\cite{Zeidler2b},
and
\cite{Zeidler1}.
Two key examples of maximally monotone operators are
continuous linear monotone operators and subdifferential operators
(in the sense of convex analysis) of
functions that are convex, lower semicontinuous, and proper.

Now let $A\colon X\To X$ be maximally monotone
and denote the associated \textbf{resolvent} by
\begin{equation}
J_A = (\Id+A)^{-1}.
\end{equation}
In \cite{Minty}, 
Minty made the seminal observation
that $J_A$ is in fact a firmly nonexpansive operator
from $X$ to $X$ and that, conversely, every firmly nonexpansive operator
arises this way:
\begin{fact}[Minty]
\label{f:Minty}
{\rm (See, e.g., \cite{Minty} or \cite{EckBer}.)}
Let $T\colon X\to X$ be firmly nonexpansive,
and let $A\colon X\To X$ be maximally monotone.
Then the following hold.
\begin{enumerate}
\item
$B= T^{-1}-\Id$ is maximally monotone (and $J_B = T$).
\item
$J_A$ is firmly nonexpansive (and $A = J_A^{-1}-\Id$).
\end{enumerate}
\end{fact}

One of the motivations to study the correspondence
between firmly nonexpansive mappings and maximally monotone operators
is the very useful correspondence
\begin{equation}
A^{-1}(0) = \Fix J_A,
\end{equation}
where $A\colon X\To X$ is maximally monotone.

From now on we assume that
\boxedeqn{
\text{$A_1,\ldots,A_m$ are maximally monotone operators on $X$,
\quad where $m\in\{2,3,\ldots\}$, }
}
that
\boxedeqn{
\label{e:lambdas}
\text{$\lambda_1,\ldots,\lambda_m$ belong to $\zeroun$
such that $\sum_{i\in I} \lambda_i = 1$,
\quad where $I=\{1,2,\ldots,m\}$, }
}
and we set
\boxedeqn{
A = \left(\sum_{i\in I}\lambda_i J_{A_i}\right)^{-1}-\Id.
}
Then the definition of the resolvent yields
\begin{equation}
\label{e:J_A}
J_A = \sum_{i\in I} \lambda_iJ_{A_i};
\end{equation}
thus,
since it is easy to see that $J_A$ is firmly nonexpansive,
it follows from Fact~\ref{f:Minty} that $A$ is maximally monotone.
We refer to the operator $A$ as the \textbf{resolvent average}
of the maximally monotone operators $A_1,\ldots,A_m$ and
we note that $J_A$ is the weighted \textbf{average of the resolvents}
$J_{A_i}$.
The operator $J_A$ is the announced generalization of the
averaged projection operator considered in \eqref{e:mproj},
and $\Fix J_A$ is the generalization of the minimizers of the function
in \eqref{e:minis}.

This introductory section is now complete. In the next section,
we shall derive an alternative description of $\Fix J_A$.

\section{The Fixed Point Set Viewed in a Product Space}

\label{s:2}

It will be quite convenient to define numbers complementary to the
convex coefficients fixed in \eqref{e:lambdas}; thus, we let
\boxedeqn{
\label{e:mus}
\mu_i = 1-\lambda_i,
\quad
\text{for every $i\in I$.}
}
Several of the results will be formulated in the Hilbert product space
\boxedeqn{
\bX = X^m, \quad
\text{ with inner product }
\scal{\bx}{\by} = \sum_{i\in I}\scal{x_i}{y_i},
}
where
$\bx = (x_i)_{i\in I}$ and $\by = (y_i)_{i\in I}$
are generic vectors in $\bX$.
The set $\bS$, defined by
\boxedeqn{
\label{e:defS}
\bS = \mmenge{\bx=(x_i)_{i\in I}\in \bX}{(\forall i\in
I)\;\;x_i=J_{\mu_i^{-1}A_i}\bigg(\sum_{j\in I\smallsetminus\{i\}}
\frac{\lambda_j}{\mu_i} x_j\bigg)},
}
turns out to be fundamental in describing $\Fix J_A$.

\begin{theorem}[correspondence between $\bS$ and $\Fix J_A$]
\label{t:SvsFix}
The operator
\begin{equation}
\label{e:defL}
L\colon \bS \to \Fix J_A \colon \bx = (x_i)_{i\in I}\mapsto
\sum_{i\in I}\lambda_i x_i
\end{equation}
is well defined, bijective, and Lipschitz continuous with constant $1$.
Furthermore, the inverse operator of $L$ satisfies
\begin{equation}
L^{-1}\colon \Fix J_A \to \bS \colon x \mapsto \big(J_{A_i}x \big)_{i\in I}
\end{equation}
and $L^{-1}$ is Lipschitz continuous with constant $\sqrt{m}$.
\end{theorem}
\begin{proof}
We proceed along several steps.

\texttt{Claim~1:} $(\forall \bx\in\bS)$ $L\bx \in \Fix J_A$
and  $\bx=\big(J_{A_i}L\bx\big)_{i\in I}$;
consequently, $L$ is well defined.\\

Let $\bx=(x_i)_{i\in I}\in\bS$ and set $\bar{x} = \sum_{i\in
I}\lambda_ix_i = L\bx$.
Using the definition of the resolvent, we have, for every $i\in I$,
\begin{subequations}
\begin{align}
\sum_{j\in I\smallsetminus\{i\}}
\frac{\lambda_j}{\mu_i} x_j \in \big(\Id + \mu_i^{-1}A_i\big)x_i
&\Leftrightarrow \sum_{j\in I\smallsetminus\{i\}}
{\lambda_j}x_j \in \mu_ix_i + A_ix_i \\
&\Leftrightarrow \sum_{j\in I\smallsetminus\{i\}}
{\lambda_j}x_j \in (1-\lambda_i)x_i + A_ix_i \\
& \Leftrightarrow \bar{x} = \sum_{j\in I}
{\lambda_j}x_j \in \big(\Id+A_i\big)x_i \\
&\Leftrightarrow x_i = J_{A_i}\bar{x} = J_{A_i}L\bx.
\end{align}
\end{subequations}
Hence $\bx = \big(J_{A_i}L\bx\big)_{i\in I}$\,, as claimed.
Moreover,
$(\forall i\in I)$ $\lambda_ix_i = \lambda_iJ_{A_i}\bar{x}$,
which, after summing over $i\in I$ and recalling \eqref{e:J_A}, yields
$\bar{x} = \sum_{i\in I}\lambda_ix_i = \sum_{i\in
I}\lambda_iJ_{A_i}\bar{x}=J_A\bar{x}$.
Thus $L\bx = \bar{x}\in\Fix J_A$ and Claim~1 is verified.

\texttt{Claim~2:} $(\forall x\in \Fix J_A)$ $\big(J_{A_i}x\big)_{i\in I}\in
\bS$.\\
Assume that $x\in\Fix J_A$ and set $(\forall i\in I)$ $y_i = J_{A_i}x$.
Then, using \eqref{e:J_A}, we see that
\begin{equation}
\label{e:0126a}
\sum_{i\in I}\lambda_i y_i = \sum_{i\in I}\lambda_i J_{A_i}x = J_Ax = x.
\end{equation}
Furthermore, for every $i\in I$,
and using \eqref{e:0126a} in the derivation of \eqref{e:0126b}
\begin{subequations}
\begin{align}
y_i = J_{A_i}x
&\Leftrightarrow x \in y_i+A_iy_i
\Leftrightarrow x-\lambda_iy_i \in \mu_iy_i+A_iy_i \\
&\Leftrightarrow \mu_i^{-1}\big(x-\lambda_iy_i\big) \in \Big(\Id +
\mu_i^{-1}A_i\Big)y_i\\
&\Leftrightarrow \mu_i^{-1}\sum_{j\in I\smallsetminus\{i\}}\lambda_jy_j \in \Big(\Id +
\mu_i^{-1}A_i\Big)y_i \label{e:0126b}\\
&\Leftrightarrow y_i=J_{\mu_i^{-1}A_i}\bigg(\sum_{j\in I\smallsetminus\{i\}}
\frac{\lambda_j}{\mu_i} y_j\bigg).
\end{align}
\end{subequations}
Thus, $(y_i)_{i\in I}\in\bS$ and Claim~2 is verified.

Having verified the two claims above, we now
turn to proving the statements announced.

First, let $x\in\Fix J_A$.
By Claim~2, $(J_{A_i}x)_{i\in I}\in\bS$.
Hence $L(J_{A_i}x)_{i\in I} = \sum_{i\in I}\lambda_iJ_{A_i}x
= J_Ax = x$ by \eqref{e:J_A}.
Thus, $L$ is surjective.

Second, assume that $\bx=(x_i)_{i\in I}$ and
$\by=(y_i)_{i\in I}$ belong to $\bS$ and that
$L\bx=L\by$.
Then, using Claim~1,
we see that
$\bx = (J_{A_i}L\bx)_{i\in I} = (J_{A_i}L\by)_{i\in I}
= \by$ and thus $L$ is injective.
Altogether, this shows that $L$ is bijective and
we also obtain the formula for $L^{-1}$.

Third, again let $\bx=(x_i)_{i\in I}$ and
$\by=(y_i)_{i\in I}$ be in $S$.
Using the convexity of $\|\cdot\|^2$, we obtain
\begin{subequations}
\begin{align}
\|L\bx-L\by\|^2 &= \Big\|\sum_{i\in I}\lambda_i(x_i-y_i)\Big\|^2
\leq\sum_{i\in I}\lambda_i\|x_i-y_i\|^2\\
&\leq\sum_{i\in I}\|x_i-y_i\|^2
=\|\bx-\by\|^2.
\end{align}
\end{subequations}
Thus, $L$ is Lipschitz continuous with constant $1$.

Finally, let $x$ and $y$ be in $\Fix J_A$.
Since $J_{A_i}$ is (firmly) nonexpansive for all $i\in I$,
we estimate
\begin{subequations}
\begin{align}
\big\|L^{-1}x-L^{-1}y\big\|^2
&=\Big\|\big(J_{A_i}x\big)_{i\in I} - \big(J_{A_i}x\big)_{i\in
I}\Big\|^2
= \sum_{i\in I}\big\| J_{A_i}x-J_{A_i}y\big\|^2\\
&\leq \sum_{i\in I}\|x-y\|^2 = m\|x-y\|^2.
\end{align}
\end{subequations}
Therefore, $L^{-1}$ is Lipschitz continuous with constant $\sqrt{m}$.
\end{proof}

\begin{remark}
Some comments regarding Theorem~\ref{t:SvsFix} are in order.
\begin{enumerate}
\item Because of the simplicity of the
bijection $L$ provided in Theorem~\ref{t:SvsFix},
the task of finding $\Fix J_A$ is essentially the same as finding
$\bS$.
\item
Note that when each $A_i$ is a normal cone operator $N_{C_i}$,
then the resolvents $J_{A_i}$ and $J_{\mu_{i}^{-1}A_i}$ simplify to
the projections $P_{C_i}$, for every $i\in I$.
\item
When $m=2$,
the set $\bS$ turns into
\begin{equation}
\bS = \menge{(x_1,x_2)\in\bX}{x_1 = J_{\lambda_2^{-1}A_1}x_2
\text{\;and\;} x_2 = J_{\lambda_1^{-1}A_2}x_1},
\end{equation}
and Theorem~\ref{t:SvsFix} coincides
with \cite[Theorem~3.6]{WangBau}.
Note that $(x_1,x_2)\in\bS$ if and only if
$x_2 \in \Fix\big(J_{\lambda_1^{-1}A_2}J_{\lambda_2^{-1}A_1}\big)$
and $x_1 = J_{\lambda_2^{-1}A_1}x_2$, which makes the connection
between the fixed point set of the composition of the two resolvents and
$\bS$.
It appears that this is a particularity of the case $m=2$;
it seems that there is no simple connection between fixed points of
$J_{\mu_m^{-1}A_m}J_{\mu_{m-1}^{-1}A_{m-1}}\cdots J_{\mu_1^{-1}A_1}$ and
$\Fix J_A$ when $m\geq 3$.
\end{enumerate}
\end{remark}

\section{Fixed Points of a Composition}

\label{s:3}

From now on, we let
\boxedeqn{
\label{e:defR}
\bR\colon\bX\to\bX\colon\bx=(x_i)_{i\in I}\mapsto
\bigg(\sum_{j\in I\smallsetminus\{i\}}
\frac{\lambda_j}{\mu_i} x_j\bigg)_{i\in I}
}
and
\boxedeqn{
\bJ\colon\bX\to\bX\colon\bx=(x_i)_{i\in I}\mapsto
\Big( J_{\mu_i^{-1}A_i}x_i\Big)_{i\in I}.
}

It is immediate from the definition of the set $\bS$ (see \eqref{e:defS}) that
\begin{equation}
\label{e:SJR}
\bS = \Fix(\bJ\circ \bR).
\end{equation}
We are thus ultimately interested in developing algorithms for finding
a fixed point of $\bJ\circ\bR$.
We start by collecting relevant information about the operator $\bR$.

\begin{proposition}
The adjoint of $\bR$ is given by
\begin{equation}
\label{e:0127a}
\bR^*\colon\bX\to\bX\colon
\bx=(x_i)_{i\in I}\mapsto
\bigg(\sum_{j\in I\smallsetminus\{i\}}
\frac{\lambda_i}{\mu_j} x_j\bigg)_{i\in I}
\end{equation}
and the set of fixed points of $\bR$ is the ``diagonal'' in $\bX$, i.e.,
\begin{equation}
\label{e:0127b}
\Fix\bR = \menge{(x)_{i\in I}\in\bX}{x\in X}.
\end{equation}
\end{proposition}
\begin{proof}
Denote the operator defined in \eqref{e:0127a} by $\bL$,
and take $\bx = (x_i)_{i\in I}$ and $\by = (y_i)_{i\in I}$ in $\bX$.
Then
\begin{subequations}
\begin{align}
\scal{\bx}{\bL\by} &= \sum_{i\in I} \scal{x_i}{(\bL\by)_i}
= \sum_{i\in I} \sum_{j\in I\smallsetminus\{i\}}
\frac{\lambda_i}{\mu_j} \scal{x_i}{y_j}\\
&= \sum_{\menge{(i,j)\in I\times I}{i\neq j}}
\frac{\lambda_i}{\mu_j}\scal{x_i}{y_j}\\
&= \sum_{j\in I} \sum_{i\in I\smallsetminus\{j\}}
\frac{\lambda_i}{\mu_j} \scal{x_i}{y_j}
= \sum_{j\in I}\scal{(\bR\bx)_j}{y_j}\\
&= \scal{\bR\bx}{\by},
\end{align}
\end{subequations}
which shows that $\bR^*=\bL$ as claimed.

Next, let $x\in X$ and denote the right side of \eqref{e:0127b} by $\bD$.
Since
\begin{equation}
\label{e:0127c-}
(\forall i\in I)\quad
\sum_{j\in I\smallsetminus\{i\}} \mu_i^{-1}\lambda_j = 1,
\end{equation}it is clear that
\begin{equation}
\label{e:0127c}
\bD\subseteq \Fix\bR.
\end{equation}
Now let $\bx = (x_i)_{i\in I}\in \Fix\bR$ and set $\bar{x}=\sum_{i\in
I}\lambda_ix_i$.
Then $\bx=\bR\bx$, i.e., for every $i\in I$, we have
\begin{subequations}
\begin{align}
x_i = (\bx)_i = (\bR\bx)_i &\Leftrightarrow
x_i =
\sum_{j\in I\smallsetminus\{i\}}
\frac{\lambda_j}{\mu_i} x_j
\Leftrightarrow
\mu_ix_i = \sum_{j\in I\smallsetminus\{i\}} \lambda_jx_j\\
&\Leftrightarrow
(1-\lambda_i)x_i = \sum_{j\in I\smallsetminus\{i\}} \lambda_jx_j
\Leftrightarrow
x_i = \sum_{j\in I} \lambda_jx_j\\
&\Leftrightarrow
x_i = \bar{x}
\end{align}
\end{subequations}
by \eqref{e:mus}.
Hence $\bx=(\bar{x})_{i\in I}\in\bD$ and thus
\begin{equation}
\label{e:0127d}
\Fix\bR \subseteq\bD.
\end{equation}
Combining \eqref{e:0127c} and \eqref{e:0127d}, we obtain \eqref{e:0127b}.
\end{proof}

\begin{remark}
If $m=2$, then $\bR^*=\bR$.
However, when $m\geq 3$, one
has the equivalence  $\bR^*=\bR$ $\Leftrightarrow$
$(\lambda_i)_{i\in I}=(\tfrac{1}{m})_{i\in I}$.
\end{remark}

The following observation will be useful when discussing
nonexpansiveness of $\bR$.

\begin{lemma}
\label{l:CS}
We have
$\displaystyle 1 \leq m\sum_{i\in I} \lambda_i^2$;
furthermore, equality holds if and only if $(\forall i\in I)$
$\lambda_i=\frac{1}{m}$.
\end{lemma}
\begin{proof}
Indeed,
\begin{align}
1 = \sum_{i\in I} \lambda_i\cdot 1
\leq \Big(\sum_{i\in I}\lambda_i^2\Big)^{1/2}\Big(\sum_{i\in I} 1^2\Big)^{1/2}
&\Leftrightarrow
1 = 1^2 \leq  \Big(\sum_{i\in I}\lambda_i^2\Big) m,
\end{align}
and the result follows from the Cauchy--Schwarz inequality and its characterization of
equality.
\end{proof}

The next result is surprising as it shows that the actual values of the convex
parameters $\lambda_i$ matter when $m\geq 3$.

\begin{proposition}[nonexpansiveness of $\bR$]
\label{p:Rne}
%Suppose that $X\neq\{0\}$.
%Then the following hold.
The following hold.
\begin{enumerate}
\item
\label{p:Rnei}
If $m=2$, then $\bR\colon (x_1,x_2)\mapsto (x_2,x_1)$;
thus, $\bR$ is an isometry and nonexpansive.
\item
\label{p:Rneii}
If $m\geq 3$, then:
$\bR$ is nonexpansive if and only if
$(\forall i\in I)$ $\lambda_i=\frac{1}{m}$, in which case $\|\bR\|=1$.
\end{enumerate}
\end{proposition}
\begin{proof}
\ref{p:Rnei}:
When $m=2$, we have $\lambda_1=\mu_2$ and $\lambda_2=\mu_1$;
thus, the definition of $\bR$ (see \eqref{e:defR}) yields the announced
formula and it is clear that then $\bR$ is an isometry and hence
nonexpansive.

\ref{p:Rneii}:
Suppose that $m\geq 3$.
Assume first that $(\forall i\in I)$ $\lambda_i = \frac{1}{m}$;
hence, $\mu_i=1-\frac{1}{m} = (m-1)/m$.
Then
\begin{equation}
\label{e:0127e}
(\forall j\in I)\quad
\lambda_j \sum_{i\in I\smallsetminus \{j\}} \frac{1}{\mu_i}
= \frac{1}{m} \sum_{i\in I\smallsetminus \{j\}} \frac{1}{(m-1)/m} = 1.
\end{equation}
Now let $\bx = (x_i)_{i\in I}\in\bX$.
Using the definition of $\bR$ (see \eqref{e:defR}),
the convexity of $\|\cdot\|^2$ in \eqref{e:0127f},
and \eqref{e:0127e} in \eqref{e:0127g}, we obtain
\begin{subequations}
\begin{align}
\|\bR\bx\|^2 &= \sum_{i\in I} \big\|(\bR\bx)_i\big\|^2
=\sum_{i\in I} \bigg\| \sum_{j\in I\smallsetminus\{i\}}
\frac{\lambda_j}{\mu_i} x_j\bigg\|^2\\
&\leq  \sum_{i\in I}  \sum_{j\in I\smallsetminus\{i\}}
\frac{\lambda_j}{\mu_i} \|x_j\|^2 \label{e:0127f}\\
&=\sum_{\menge{(i,j)\in I\times I}{i\neq j}}
\frac{\lambda_j}{\mu_i}\|x_j\|^2\\
&=\sum_{j\in I} \lambda_j\|x_j\|^2 \sum_{i\in I\smallsetminus\{j\}}
\frac{1}{\mu_i}\\
&=\sum_{j\in I}\|x_j\|^2 \label{e:0127g}\\
&=\|\bx\|^2.
\end{align}
\end{subequations}
Since $\bR$ is linear, it follows that $\bR$ is nonexpansive;
furthermore, since $\Fix\bR\neq\{0\}$ by \eqref{e:0127b}, we then
have $\|\bR\|=1$.

To prove the remaining implication, we demonstrate the contrapositive
and thus assume that
\begin{equation}
\label{e:0127h}
\big(\lambda_i\big)_{i\in I} \neq \big(\tfrac{1}{m}\big)_{i\in I}.
\end{equation}
Take $u\in X$ such that $\|u\|=1$ and set
$(\forall i\in I)$ $x_i = \mu_i u$ and
$\bx = (x_i)_{i\in I}$.
We compute
\begin{subequations}
\begin{align}
\|\bx\|^2 &= \sum_{i\in I}\|x_i\|^2 =
\sum_{i\in I}\|\mu_iu\|^2 = \sum_{i\in I} \mu_i^2\\
&= \sum_{i\in I} (1-\lambda_i)^2
= \sum_{i \in I} \big(1 - 2\lambda_i +\lambda_i^2\big)\\
&=m-2 + \sum_{i\in I}\lambda_i^2.
\end{align}
\end{subequations}
Using \eqref{e:0127a},
the fact that $\|u\|=1$,
we obtain
\begin{subequations}
\begin{align}
\|\bR^*\bx\|^2 &= \sum_{i\in I} \lambda_i^2\bigg\|\sum_{j\in
I\smallsetminus\{i\}} \mu_j^{-1}x_j\bigg\|^2
=  \sum_{i\in I} \lambda_i^2\bigg\|\sum_{j\in
I\smallsetminus\{i\}} \mu_j^{-1}\mu_ju\bigg\|^2\\
&= \sum_{i\in I} \lambda_i^2 \big\|(m-1)u\big\|^2
= (m-1)^2\sum_{i\in I}\lambda_i^2
\end{align}
\end{subequations}
Altogether,
\begin{subequations}
\begin{align}
\|\bR^*\bx\|^2-\|\bx\|^2 &=
2-m + \big((m-1)^2 -1\big)\sum_{i\in I}\lambda_i^2\\
&= (m-2)\Big(-1 + m\sum_{i\in I}\lambda_i^2\Big).
\end{align}
\end{subequations}
Now $m\geq 3$ implies that $m-2>0$; furthermore, by \eqref{e:0127h} and
Lemma~\ref{l:CS}, $-1+m\sum_{i\in I}\lambda_i^2 > 0$.
Therefore,
\begin{equation}
\|\bR^*\bx\|>\|\bx\|.
\end{equation}
This implies $\|\bR^*\|>1$ and hence $\|\bR\|>1$ by
\cite[Theorem~3.9-2]{Krey}.
Since $\bR$ is linear, it cannot be nonexpansive.
\end{proof}

For algorithmic purposes, nonexpansiveness is a desirable property but it
does not guarantee the convergence of the iterates
to a fixed point (consider, e.g., $-\Id$).
The very useful notion of an averaged mapping, which is
intermediate between nonexpansiveness and firm nonexpansiveness,
was introduced by Baillon, Bruck, and Reich in \cite{BBR}.

\begin{definition}[averaged mapping]
Let $T\colon X\to X$. Then $T$ is \textbf{averaged}
if there exist a nonexpansive mapping $N\colon X\to X$ and
$\alpha\in\left[0,1\right[$ such that
\begin{equation}
T = (1-\alpha)\Id + \alpha N;
\end{equation}
if we wish to emphasis the constant $\alpha$, we say that
$T$ is \textbf{$\alpha$-averaged}.
\end{definition}

It is clear from the definition that every averaged mapping is
nonexpansive;
the converse, however, is false: indeed,
$-\Id$ is nonexpansive, but not averaged.
It follows from Fact~\ref{f:firm}
that every firmly nonexpansive mapping is $\tfrac{1}{2}$-averaged.

The class of averaged mappings is closed under compositions;
this is not true for firmly nonexpansive mappings: e.g., consider
two projections
onto two lines that meet at $0$ at a $\pi/4$ angle.
Let us record the following well known key properties.

\begin{fact}
\label{f:av}
Let $T$, $T_1$, and $T_2$ be mappings from $X$ to $X$,
let $\alpha_1$ and $\alpha_2$ be in $\left[0,1\right[$,
and let $x_0\in X$.
Then the following hold.
\begin{enumerate}
\item
\label{f:avi}
$T$ is firmly nonexpansive if and only if $T$ is $\tfrac{1}{2}$-averaged.
\item
\label{f:avii}
If $T_1$ is $\alpha_1$-averaged and $T_2$ is $\alpha_2$-averaged,
then $T_1\circ T_2$ is $\alpha$-averaged, where
\begin{equation}
\alpha = \begin{cases}
0, &\text{if $\alpha_1=\alpha_2=0$;}\\
\displaystyle \frac{2}{1 + 1/\max\{\alpha_1,\alpha_2\}}, &\text{otherwise}
\end{cases}
\end{equation}
is the harmonic mean of $1$ and $\max\{\alpha_1,\alpha_2\}$.
\item
\label{f:aviii}
If $T_1$ and $T_2$ are averaged, and $\Fix(T_1\circ T_2)\neq\emp$, then
$\Fix(T_1\circ T_2) = \Fix(T_1)\cap\Fix(T_2)$.
%\hl{Perhaps not needed.}
\item
\label{f:aviv}
If $T$ is averaged and $\Fix T\neq\emp$, then the sequence of iterates
$(T^nx_0)_\nnn$ converges weakly\footnote{When $T$ is firmly
nonexpansive, the weak convergence goes back at least to
\cite{Browder67}.} to a point in $\Fix T$; otherwise,
$\|T^nx_0\|\to\pinf$.
\end{enumerate}
\end{fact}
\begin{proof}
\ref{f:avi}: This is well known and immediate from Fact~\ref{f:firm}.

\ref{f:avii}: The fact that the composition of averaged mappings is again
averaged is well known and implicit in the proof of
\cite[Corollary~2.4]{BBR}. For the exact constants, see
\cite[Lemma~2.2]{Comb04} or \cite[Proposition~4.32]{BC2011}.

\ref{f:aviii}:
This follows from \cite[Proposition~1.1, Proposition~2.1, and Lemma~2.1]{BR}.
See also \cite[Theorem~3]{Opial67} for the case when $\Fix T\neq\emp$.

\ref{f:aviv}:
This follows from \cite[Corollary~1.3 and Corollary~1.4]{BR}.
\end{proof}

\begin{theorem}[averagedness of $\bR$]
\label{t:Rav}
The following hold.
\begin{enumerate}
\item
\label{t:Ravi}
If $m=2$, then $\bR$ is not averaged.
\item
\label{t:Ravii}
If $m\geq 3$ and $(\forall i\in I)$ $\lambda_i=\tfrac{1}{m}$,
then $\bR = (1-\alpha)\Id + \alpha\bN$, where
$\alpha=\tfrac{m}{2m-2}$ and $\bN$ is an isometry; in particular,
$\bR$ is $\alpha$-averaged.
\end{enumerate}
\end{theorem}
\begin{proof}
\ref{t:Ravi}:
Assume that $m=2$.
By Proposition~\ref{p:Rne}\ref{p:Rnei},
$\bR\colon (x_1,x_2)\mapsto(x_2,x_1)$.
We argue by contradiction and thus assume that
$\bR$ is averaged, i.e.,
there exist a nonexpansive mapping $\bN\colon\bX\to\bX$
and $\alpha\in\left[0,1\right[$ such that
$\bR = (1-\alpha)\Id + \alpha\bN$. Since $\bR\neq\Id$,
it is clear that $\alpha>0$. Thus,
\begin{equation}
\bN\colon\bX\to\bX\colon
(x_1,x_2)\mapsto \alpha^{-1}
\big( x_2-x_1+\alpha x_1,x_1-x_2+\alpha x_2\big).
\end{equation}
Now take $u\in X$ such that $\|u\|=1$ and
set $\bx = (x_1,x_2) = (0,\alpha u)$.
Then
\begin{equation}
\|\bx\|^2 = \|0\|^2 + \|\alpha u\|^2 = \alpha^2
\end{equation}
and $\bN\bx = \big(u,(\alpha-1)u\big)$.
Thus,
\begin{equation}
\|\bN\bx\|^2 = \|u\|^2 + \|(\alpha-1)u\|^2 = 1+ (1-\alpha)^2 = \alpha^2 +
2(1-\alpha) >\alpha^2 =
\|\bx\|^2.
\end{equation}
Hence $\|\bN\|>1$ and, since $\bN$ is linear, $\bN$ cannot be nonexpansive.
This contradiction completes the proof of \ref{t:Ravi}.

\ref{t:Ravii}:
Assume that $m \geq 3$ and that $(\forall i\in I)$
$\lambda_i=\tfrac{1}{m}$.
For future reference, we observe that
\begin{equation}
\label{e:0128a}
(\forall i\in I)(\forall j\in I)
\quad
\frac{\lambda_j}{\mu_i} = \frac{\tfrac{1}{m}}{1-\tfrac{1}{m}}=\frac{1}{m-1}.
\end{equation}
We start by defining
\begin{equation}
L\colon\bX\to X\colon (x_i)_{i\in I}\mapsto \sum_{i\in I} x_i.
\end{equation}
Then it is easily verified that
\begin{equation}
L^* \colon X\to\bX\colon x\mapsto (x)_{i\in I}
\end{equation}
and hence that
\begin{equation}
\label{e:0128c}
L^*LL^*L = mL^*L.
\end{equation}
Now set
\begin{equation}
\alpha = \frac{m}{2m-2}
\quad\text{and}\quad
\bN = \alpha^{-1}\big( \bR - (1-\alpha)\Id\big).
\end{equation}
Then $\alpha\in\zeroun$ and $\bR = \alpha\bN + (1-\alpha)\Id$;
thus, it suffices to show that $\bN$ is an isometry.
Note that
\begin{equation}
\label{e:0128b}
\alpha -1 = -\alpha + \frac{1}{m-1}.
\end{equation}
Take $\bx = (x_i)_{i\in I}\in\bX$.
Using \eqref{e:defR}, \eqref{e:0128a}, and \eqref{e:0128b},
we obtain for every $i\in I$,
\begin{subequations}
\begin{align}
(\bN\bx)_i
&= \alpha^{-1}\big(- (1-\alpha)x_i+ (\bR\bx)_i \big)\\
&= \alpha^{-1}\bigg((\alpha-1)x_i+ \sum_{j\in I\smallsetminus\{i\}}
\frac{\lambda_j}{\mu_i} x_j \bigg)\\
&= \alpha^{-1}\bigg(-\alpha x_i +\frac{1}{m-1}x_i+ \sum_{j\in I\smallsetminus\{i\}}
\frac{1}{m-1} x_j \bigg)\\
&= \alpha^{-1}\bigg(-\alpha x_i + \sum_{j\in I} \frac{1}{m-1} x_j \bigg)\\
&= -x_i + \frac{\alpha^{-1}}{m-1}L\bx\\
&= -x_i + \frac{2}{m}L\bx;
\end{align}
\end{subequations}
hence, $\bN\bx = -\bx + \frac{2}{m}L^*L\bx$.
It follows that
$\bN = -\Id + \frac{2}{m}L^*L$ and thus $\bN^*=\bN$.
Using \eqref{e:0128c}, we now obtain
\begin{subequations}
\begin{align}
\bN^*\bN&=\bN\bN = \big(-\Id + \tfrac{2}{m}L^*L\big)\big(-\Id +
\tfrac{2}{m}L^*L\big)\\
&= \Id - \tfrac{2}{m}L^*L - \tfrac{2}{m}L^*L +
\tfrac{4}{m^2}\big(L^*LL^*L\big) \\
&= \Id - \tfrac{4}{m}L^*L + \tfrac{4}{m^2}\big(mL^*L\big)\\
&=\Id.
\end{align}
\end{subequations}
Therefore, $\|\bN\bx\|^2 = \scal{\bN\bx}{\bN\bx} =
\scal{\bx}{\bN^*\bN\bx} = \scal{\bx}{\bx}=\|\bx\|^2$
and hence $\bN$ is an isometry; in particular,
$\bN$ is nonexpansive
and $\bR$ is $\alpha$-averaged.
%Therefore,
%by \cite[Theorem~3.9-4(e)]{Krey},
%$\|\bN\| = \sqrt{\|\bN\bN^*\|} = \sqrt{\|\Id\|} = 1$
%and $\bN$ is nonexpansive as claimed.
\end{proof}

We are now in a position to describe the set $\bS$
as the fixed point set of an averaged mapping.

\begin{corollary}
\label{c:punch}
Suppose that $m\geq 3$ and that $(\forall i\in I)$
$\lambda_i=\tfrac{1}{m}$. Then
$\bJ\circ\bR$ is $\tfrac{2m}{3m-2}$-averaged and
$\Fix(\bJ\circ\bR)=\bS$.
\end{corollary}
\begin{proof}
On the one hand, since $\bJ$ is clearly firmly nonexpansive,
$\bJ$ is $\tfrac{1}{2}$-averaged.
On the other hand, by Theorem~\ref{t:Rav}\ref{t:Ravii},
$\bR$ is $\tfrac{m}{2m-2}$-averaged.
Since $0<\tfrac{1}{2}<\tfrac{m}{2m-2}$, it follows from
Fact~\ref{f:av}\ref{f:avii} that $\bJ\circ\bR$ is
$\alpha$-averaged, where
\begin{equation}
\alpha = \frac{2}{1+1/(m/(2m-2))} = \frac{2m}{3m-2},
\end{equation}
as claimed. To complete the proof, recall \eqref{e:SJR}.
\end{proof}

\section{Two New Algorithms}

\label{s:2new}

In Section~\ref{s:2}, we saw that $\Fix J_A = L(\bS)$ (see
Theorem~\ref{t:SvsFix}), and in Section~\ref{s:3} we discovered
that $\bS=\Fix(\bJ\circ\bR)$ is the fixed point set of an averaged
operator. This analysis leads to new algorithms for finding
a point in $\Fix J_A$.

\begin{theorem}
\label{t:rigalg}
Suppose that $m\geq 3$ and that $(\forall i\in I)$
$\lambda_i=\tfrac{1}{m}$.
Let $\bx_0 = (x_{0,i})_{i\in I}\in\bX$ and generate the sequence
$(\bx_n)_\nnn$ by
\begin{equation}
\label{e:rigalg}
(\forall\nnn)\quad
\bx_{n+1} = (\bJ\circ\bR)\bx_n.
\end{equation}
Then exactly one of the following holds.
\begin{enumerate}
\item
\label{t:rigalgi}
$\Fix J_A\neq\emp$, $(\bx_n)_\nnn$ converges weakly to a point
$\bx = (x_i)_{i\in I}$ in $\bS$
and $(\sum_{i\in I} \lambda_i\bx_{n,i})_\nnn$ converges weakly
to $\sum_{i\in I} \lambda_ix_i\in\Fix J_A$.
\item
\label{t:rigalgii}
$\Fix J_A=\emp$ and $\|\bx_n\|\to\pinf$.
\end{enumerate}
\end{theorem}
\begin{proof}
By Theorem~\ref{t:SvsFix}, $\bS\neq\emp$ if and only if $\Fix J_A\neq\emp$.
Furthermore, Corollary~\ref{c:punch} shows that
$\bS = \Fix(\bJ\circ\bR)$, where $\bJ\circ\bR$ is averaged.
The result thus follows from Fact~\ref{f:av}\ref{f:aviv},
Theorem~\ref{t:SvsFix}, and the weak continuity of the operator $L$ defined
in \eqref{e:defL}.
\end{proof}

\begin{remark}
The assumption that $m\geq 3$ in Theorem~\ref{t:rigalg} is critical:
indeed, suppose that $m=2$.
Then, by Proposition~\ref{p:Rne}\ref{p:Rnei}, $\bR\colon (x_1,x_2)\mapsto
(x_2,x_1)$. Now assume further that $A_1=A_2\equiv 0$.
Then $\bJ=\Id$ and hence $\bJ\circ\bR=\bR$. Thus, if
$y$ and $z$ are two distinct points in $X$ and the sequence
$(\bx_n)_\nnn$ is generated by iterating
$\bJ\circ\bR$ with a starting point $\bx_0 = (y,z)$,
then
\begin{equation}
(\forall\nnn)\quad
\bx_{n} = \begin{cases}
(y,z), &\text{if $n$ is even;}\\
(z,y), &\text{if $n$ is odd.}
\end{cases}
\end{equation}
Consequently,
$(\bx_n)_\nnn$ is a bounded sequence that is not weakly convergent.
On the other hand, keeping the assumption $m=2$ but allowing again for
general maximally monotone operators $A_1$ and $A_2$, and assuming
that $\Fix J_A\neq\emp$,
we observe that
\begin{equation}
\bJ\circ\bR\circ\bJ\circ\bR\colon
\bX\to\bX\colon
(x_1,x_2)\mapsto \big( J_{\lambda_2^{-1}A_1}J_{\lambda_1^{-1}A_2}x_1,
 J_{\lambda_1^{-1}A_2}J_{\lambda_2^{-1}A_1}x_2\big).
\end{equation}
Hence, by \cite[Theorem~5.3]{WangBau},
the \emph{even} iterates of $\bJ\circ\bR$ will converge weakly to
point $(\bar{x}_1,\bar{x}_2)$
with $\bar{x}_1 =  J_{\lambda_2^{-1}A_1}J_{\lambda_1^{-1}A_2}\bar{x}_1$ and
$\bar{x}_2 =  J_{\lambda_1^{-1}A_2}J_{\lambda_2^{-1}A_1}\bar{x}_2$.
However, $(\bar{x}_1,\bar{x}_2)\notin\bS$ in general.
\end{remark}

Just as the Gauss-Seidel iteration can
be viewed as a modification of the Jacobi iteration
where new information is immediately utilized
(see, e.g.,  \cite[Section~4.1]{Saad}),
we shall propose a similar modification of
the iteration of the operator $\bJ\circ\bR$ analyzed above.
To this end, we introduce,
for every $k\in I$,
the following operators from $\bX$ to $\bX$:
\boxedeqn{
\label{e:defRk}
\big(\forall \bx=(x_i)_{i\in I}\in\bX\big)(\forall i\in I)\quad
(\bR_k\bx)_i =
\begin{cases}
x_i, &\text{if $i\neq k$;}\\[+4 mm]
\displaystyle \sum_{j\in I\smallsetminus\{k\}} \frac{\lambda_j}{\mu_k} x_j,
&\text{if $i=k$,}
\end{cases}
}
and
\boxedeqn{
\label{e:defJk}
\big(\forall \bx=(x_i)_{i\in I}\in\bX\big)(\forall i\in I)\quad
(\bJ_k\bx)_i =
\begin{cases}
x_i, &\text{if $i\neq k$;}\\[+4 mm]
\displaystyle J_{\mu_k^{-1}A_k}x_k,
&\text{if $i=k$.}
\end{cases}
}
It follows immediately from the definition of $\bS$
(see \eqref{e:defS}) that
\begin{equation}
\bS = \bigcap_{k\in I} \Fix(\bJ_k\circ\bR_k).
\end{equation}
This implies
\begin{equation}
\label{e:0128d}
\bS \subseteq \Fix\big(\bJ_m\circ\bR_m\circ\cdots\circ
\bJ_1\circ\bR_1\big),
\end{equation}
and it motivates---but does not justify---to iterate
the composition
\boxedeqn{
\bT = \bJ_m\circ\bR_m\circ\cdots\circ
\bJ_1\circ\bR_1
}
in order to find points in $\bS$.

\begin{remark}
\label{r:bloodyback}
In general, the composition
$\bT = \bJ_m\circ\bR_m\circ\cdots\circ
\bJ_1\circ\bR_1$
is not nonexpansive:
indeed, assume that $(\forall k\in I)$ $A_k\equiv 0$
so that $\bJ_k=\Id$.
Then $\bT = \bR_m\circ\cdots\circ \bR_1$ and we show in
Appendix~B that this composition is not nonexpansive
and neither is any $\bR_k$.
\end{remark}

\begin{remark}
\label{r:bT}
One may verify that $\bJ_k\circ\bR_k$ is Lipschitz continuous with
constant $\sqrt{m/(m-1)}$ when $\big(\lambda_i\big)_{i\in I} =
\big(\tfrac{1}{m}\big)_{i\in I}$ (see Appendix~B).
In turn, this implies that
\begin{equation}
\bT \text{ is Lipschitz continuous with constant }
\Big(\frac{m}{m-1}\Big)^{m/2}.
\end{equation}
As $m\to\pinf$, the Lipschitz constant of $\bT$ decreases to
$\sqrt{\exp(1)}\approx 1.6487$.
\end{remark}

\begin{remark}[numerical experiments]
In our numerical experiments, we assumed that 
$X = \RR^{50}$, that $m=55$, and that $(\lambda_i)_{i\in I}=
(\tfrac{1}{m})_{i\in I}$. 
We considered $m$ hyperplanes and the associated 
normal cone operators; this corresponds to 
a mildly overdetermined system of linear equations and 
to resolvents that are projection mappings 
$(P_i)_{i\in I}$. 
As the aim is to find fixed points of the 
 the averaged resolvent $J_A$, 
which in this case is the (equally weighted) average of
the projections $(P_i)_{i\in I}$ (see \eqref{e:mproj} and
\eqref{e:J_A}),
we measured 
performance at the $n$ iteration of $x_n\in X$ by  
the relative error function in decibel (dB), 
i.e., by 
\begin{equation}\label{e:relprox}
10 \log_{10} \bigg( \frac{{\| J_Ax_n - x_n \|}^2}{{\| J_Ax_0 - x_0
\|}^2} \bigg).
\end{equation}
For all experiments, the starting point $x_0$ is the zero vector.
We compared three algorithms denoted 
$\algJA$, $\algJR$, and $\algT$, 
which correspond to iterating 
$J_A$, $\bJ \circ \bR$, and $\bT$, respectively.
The last two new algorithms operate in the product space $\bX$;
thus, we project the $n$th iterate down to $X$ via
$(x_n)_{i\in I} = (x_{n,i})_{i\in I}\mapsto \sum_{i\in
I}\lambda_ix_{n,i}$ to compare to $\algJA$. 
The random sets (i.e., the hyperplanes) 
were generated in 5 instances,
and the values of \eqref{e:relprox} were averaged for
each iteration number.  
These values are plotted in Figure~\ref{fig:1}.
\begin{figure}[H]
\centering
{\large \scalebox{0.7}{\input{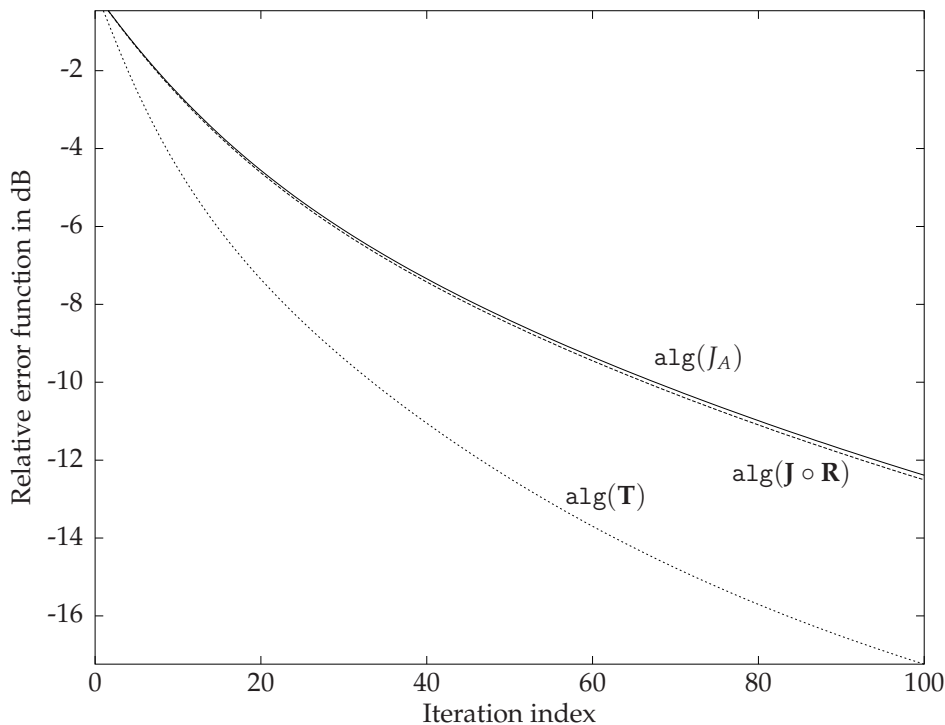}}}
\caption{Values of the relative error function for the three algorithms.}
\label{fig:1}
\end{figure}
As seen in Figure~\ref{fig:1}, 
the new rigorous algorithm $\algJR$
performs better than $\algJA$,
although the performance gain is slight.
(Convergence is guaranteed by 
Fact~\ref{f:av}\ref{f:aviv} and Theorem~\ref{t:rigalg}\ref{t:rigalgi}.)
Furthermore, the new heuristic algorithm $\algT$,
which currently lacks a convergence analysis (see
Remark~\ref{r:bloodyback}) substantially outperfoms
$\algJR$.
\end{remark}

Let us now list some open problems.

\begin{remark}[open problems]
\label{r:op}
Suppose that $m\geq 3$.
We do not know the answers to the following questions.
\begin{itemize}
\item[\textbf{Q1:}] Concerning \eqref{e:0128d}, is it actually true that
\begin{equation}
\bS \stackrel{}{=} \Fix\bT = \Fix\big(\bJ_m\circ\bR_m\circ\cdots\circ
\bJ_1\circ\bR_1\big)\,?
\end{equation}
\item[\textbf{Q2:}] Can one give simple sufficient or necessary
conditions for the convergence of the heuristic algorithm, i.e.,
the iteration of $\bT$, when $\Fix\bT\neq\emp$?
\item[\textbf{Q3:}] Under the most general assumption \eqref{e:lambdas},
we observed convergence in numerical experiments of the new rigorous
algorithm even though
there is no underlying theory---see Proposition~\ref{p:Rne}\ref{p:Rneii}
and Theorem~\ref{t:rigalg}. Can one provide simple sufficient
or necessary conditions for the convergence of the sequence defined
by \eqref{e:rigalg}?
\end{itemize}
\end{remark}

\begin{remark}
Concerning Remark~\ref{r:op},
we note that the first two questions posed  have
affirmative answers when $m=2$.
Indeed, one then computes
\begin{equation}
\bT\colon\bX\to\bX\colon
(x_1,x_2)\mapsto
\Big( J_{\lambda_2^{-1}A_1}x_2,
 J_{\lambda_1^{-1}A_2} J_{\lambda_2^{-1}A_1}x_2\Big)
\end{equation}
and hence $(x_1,x_2)\in\Fix\bT$ if and only if
$x_1 = J_{\lambda_2^{-1}A_1}x_2$ and
$x_2 = J_{\lambda_1^{-1}A_2}x_1$,
which is the same as requiring that
$(x_1,x_2)\in\bS$.
When $\bS=\Fix \bT\neq\emp$, then
the iterates of $\bT$ converge weakly to a fixed point by
\cite[Theorem~5.3(i)]{WangBau}.
\end{remark}

\section*{Appendix~A}

Most of this part of the appendix is part of the folklore;
however,
we include it here for completeness and because we have not quite
found a reference that makes all points we wish to stress.

We assume that $m\in\{1,2,\ldots\}$, that $I = \{1,2,\ldots,m\}$ and
that $(C_i)_{i\in I}$ is a family of closed hyperplanes given by
\begin{equation}
(\forall i\in I)
\qquad
C_i = \menge{x\in X}{\scal{a_i}{x}=b_i},
\quad
\text{where }
a_i\in X\smallsetminus\{0\}
\text{ and }
b_i\in\RR,
\end{equation}
with corresponding projections $P_i$.
Set $A\colon X\to\RR^m\colon x\mapsto \scal{a_i}{x}$ and
$b = (b_i)_{i\in I}\in\RR^m$.
Then $A^*\colon\RR^m\to X\colon (y_i)_{i\in I}\mapsto \sum_{i\in I}
y_i a_i$.
Denote, for every $i\in I$, the $i^\text{th}$ unit vector in
$\RR^m$ by $e_i$, and the projection $\RR^m\to\RR^m\colon
y\mapsto \scal{y}{e_i}e_i$ onto $\RR\, e_i$ by $Q_i$.
Note that
\begin{equation}
\label{e:0129b}
\sum_{i\in I} Q_i = \Id
\quad\text{and}\quad
(\forall (i,j)\in I\times I)
\;\;
Q_iQ_j = \begin{cases}
Q_i, &\text{if $i=j$;}\\
0, &\text{otherwise.}
\end{cases}
\end{equation}
We now assume that
\begin{equation}
\label{e:0129c}
(\forall i\in I)\quad
\|a_i\|=1,
\end{equation}
which gives rise to the pleasant representation
of the projectors as
\begin{equation}
\label{e:0129a}
(\forall i\in I)\quad
P_i \colon x\mapsto x - A^*Q_i(Ax-b)
\end{equation}
and to (see \eqref{e:minis}) 
\begin{equation}
(\forall x\in X)\quad
\|Ax-b\|^2 = 
\sum_{i\in I} \big|\scal{a_i}{x}-b_i\big|^2 = 
\sum_{i\in I} \big\|x-P_ix\big\|^2. 
\end{equation}
Now let $x\in X$.
Using \eqref{e:0129a} and \eqref{e:0129b}, we thus
obtain the following characterization
of fixed points of averaged projections:
\begin{subequations}
\label{e:0203a}
\begin{align}
& x \in\Fix\Big(\sum_{i\in I}\lambda_i P_i\Big)\\
&\Leftrightarrow
x = \sum_{i\in I}\lambda_iP_ix\\
&\Leftrightarrow
x = \sum_{i\in I}\lambda_i\big(  x - A^*Q_i(Ax-b)\big)\\
&\Leftrightarrow
x = \Big(\sum_{i\in I}  \lambda_ix \Big)-
A^*\sum_{i\in I}\lambda_iQ_i(Ax-b)\\
&\Leftrightarrow
A^*\Big(\sum_{i\in I}\lambda_iQ_i\Big)(Ax-b) = 0\\
&\Leftrightarrow
A^*\Big(\sum_{i\in I}\sqrt{\lambda_i}Q_i\Big)\Big(\sum_{i\in
I}\sqrt{\lambda_i}Q_i\Big)(Ax-b) = 0\\
&\Leftrightarrow
A^*\Big(\sum_{i\in I}\sqrt{\lambda_i}Q_i\Big)\Big(\sum_{i\in
I}\sqrt{\lambda_i}Q_i\Big)Ax =
A^*\Big(\sum_{i\in I}\sqrt{\lambda_i}Q_i\Big)\Big(\sum_{i\in
I}\sqrt{\lambda_i}Q_i\Big)b\\
&\Leftrightarrow
\bigg(\Big(\sum_{i\in I}\sqrt{\lambda_i}Q_i\Big)A\bigg)^*
\bigg(\Big(\sum_{i\in I}\sqrt{\lambda_i}Q_i\Big)A\bigg)x =
\bigg(\Big(\sum_{i\in I}\sqrt{\lambda_i}Q_i\Big)A\bigg)^*\Big(\sum_{i\in
I}\sqrt{\lambda_i}Q_i\Big)b\\
&\Leftrightarrow
\text{$x$ satisfies the normal equation of the system}\\
&\qquad \bigg(\Big(\sum_{i\in
I}\sqrt{\lambda_i}Q_i\Big)A\bigg)x = \Big(\sum_{i\in
I}\sqrt{\lambda_i}Q_i\Big)b\\
&\Leftrightarrow
\bigg(\Big(\sum_{i\in I}\sqrt{m\lambda_i}Q_i\Big)A\bigg)^*
\bigg(\Big(\sum_{i\in I}\sqrt{m\lambda_i}Q_i\Big)A\bigg)x\\
&\qquad =
\bigg(\Big(\sum_{i\in I}\sqrt{m\lambda_i}Q_i\Big)A\bigg)^*\Big(\sum_{i\in
I}\sqrt{m\lambda_i}Q_i\Big)b\\
&\Leftrightarrow
\text{$x$ satisfies the normal equation of the system}\\
&\qquad \bigg(\Big(\sum_{i\in
I}\sqrt{m\lambda_i}Q_i\Big)A\bigg)x = \Big(\sum_{i\in
I}\sqrt{m\lambda_i}Q_i\Big)b.
\end{align}
\end{subequations}
Note that when $(\lambda_i)_{i\in I}=(\tfrac{1}{m})_{i\in I}$, i.e.,
we have equal weights, then \eqref{e:0203a} and \eqref{e:0129b} yield
\begin{subequations}
\begin{align}
x \in\Fix\bigg(\frac{1}{m}\sum_{i\in I} P_i\bigg)
&\Leftrightarrow
\bigg(\Big(\sum_{i\in I}Q_i\Big)A\bigg)^*
\bigg(\Big(\sum_{i\in I}Q_i\Big)A\bigg)x =
\bigg(\Big(\sum_{i\in I}Q_i\Big)A\bigg)^*\Big(\sum_{i\in
I}Q_i\Big)b\\
&\Leftrightarrow
A^*Ax = A^*b\\
&\Leftrightarrow
\text{$x$ satisfies the normal equation of the system $Ax=b$}\\
&\Leftrightarrow
\text{$x$ is a least squares solution of the system $Ax=b$.}
\end{align}
\end{subequations}
\emph{In other words, the fixed points of the equally
averaged projections onto hyperplanes are precisely the classical
\textbf{least squares solutions} encountered in linear algebra,
i.e., the solutions to the classical
\textbf{normal equation} $A^*Ax=A^*b$ of the system $Ax=b$.}
The idea of least squares solutions goes back to the famous prediction
of the asteroid \emph{Ceres} due to Carl Friedrich Gauss in 1801
(see \cite[Subsection~1.1.1]{Bjorck} and
also \cite[Epilogue in Section~4.6]{Meyer}).

\noindent
\textbf{Example.}
Consider the following inconsistent linear system of equations
\begin{subequations}
\label{e:Byrne}
\begin{align}
x&=1\\
x&=2,
\end{align}
\end{subequations}
which was also studied by Byrne \cite[Subsection~8.3.2 on
page~100]{Byrne08}.
Here $m=2$ and \eqref{e:0129c} holds, and the above discussion yields
that $\Fix\big(\tfrac{1}{2}P_1+\tfrac{1}{2}P_2\big)$ and the set of least
squares solutions coincide, namely with the singleton
$\big\{\tfrac{3}{2}\big\}$.
Now change the representation to
\begin{subequations}
\begin{align}
2x&=2\\
x&=2,
\end{align}
\end{subequations}
so that \eqref{e:0129c} is violated. The set of fixed points remains
unaltered as the two hyperplanes $C_1$ and $C_2$ are unchanged and
thus it equals $\big\{\tfrac{3}{2}\big\}$.
However, the set of least squares solutions is now
$\big\{\tfrac{6}{5}\big\}$.
Similarly and returning to the first representation
in \eqref{e:Byrne},
the set of fixed points will changes if we consider different weights,
say $\lambda_1=\tfrac{1}{3}$ and $\lambda_2=\tfrac{2}{3}$:
indeed, we then obtain $\Fix\big(\tfrac{1}{3}P_1+\tfrac{2}{3}P_2\big)
=\big\{\tfrac{5}{3}\big\}$ while the set of least squares solutions
is still
$\big\{\tfrac{3}{2}\big\}$.

\section*{Appendix~B}

The proof of the following result is simple and hence omitted.

\noindent
\textbf{Lemma~B.1}
Let $(\alpha_{i,j})_{(i,j)\in I\times I}$ and
$(\beta_{i,j})_{(i,j)\in I\times I}$ be in $\RR^{m\times m}$,
and define
\begin{equation}
\bA\colon\bX\to\bX\colon (x_i)_{\in I} \mapsto
\Big(\sum_{j\in I}\alpha_{i,j}x_j\Big)_{i\in I}
\text{~and~}
\bB\colon\bX\to\bX\colon (x_i)_{\in I} \mapsto
\Big(\sum_{j\in I}\beta_{i,j}x_j\Big)_{i\in I}\,. 
\end{equation}
Then
\begin{equation}
\bA\circ\bB\colon
\bX\to\bX\colon
(x_i)_{\in I} \mapsto
\Big(\sum_{j\in I}\gamma_{i,j}x_j\Big)_{i\in I}\;,
\end{equation}
where
$(\forall (i,j)\in I\times I)$
$\gamma_{i,j} = \sum_{k\in I} \alpha_{i,k}\beta_{k,j}$.
Furthermore, the following hold:
\begin{enumerate}
\item
If for every $i\in I$,
$\sum_{j\in I} \alpha_{i,j}=1=\sum_{j\in I}\beta_{i,j}$,
then $\sum_{j\in I} \gamma_{i,j}=1$ as well.
\item
If for every $(i,j)\in I\times I$,
$\alpha_{i,j}\geq 0$ and $\beta_{i,j}\geq 0$,
then $\gamma_{i,j}\geq 0$ as well.
\end{enumerate}

\noindent
\textbf{Proof of Remark~\ref{r:bloodyback}.}
No $\bR_k$ is nonexpansive and neither is
$\bR_{m}\circ \cdots \circ \bR_{2}\circ \bR_{1}$.

\begin{proof}
Take $\bx=(x_i)_{i\in I}\in\bX$,
and let $i\in I$.
If $i\neq k$, then $(\bR_{k}\bx)_{i} = x_i$;
otherwise, $i=k$ and $(\bR_{k}\bx)_{k}$ is a convex
combination of the vectors $\{x_j\}_{j\in I\smallsetminus\{k\}}$.
In either case,  $(\bR_{k}\bx)_{k}$ is a convex
combination of the vectors $\{x_j\}_{j\in I\smallsetminus\{k\}}$.
Thus if $u\in X$ satisfies $\|u\|=1$ and
\begin{equation}
\label{e:0201c}
(\forall i\in I)\quad
x_i = \begin{cases}
u, &\text{if $i\neq k$;}\\
0, &\text{if $i=k$,}
\end{cases}
\end{equation}
then $\bR_k\bx = (u)_{i\in I}$ and hence
$\|\bR_k\bx\|^2 = \sum_{i\in I}\|u\|^2 = m > m-1 = \sum_{j\in
I\smallsetminus\{k\}}\|u\|^2 = \|\bx\|^2$.
Therefore $\bR_k$ is not nonexpansive.

Now assume that $k=1$ and that $\bx$ is defined as in
\eqref{e:0201c}.
The above reasoning shows that $\bR_1\bx = (u)_{i\in I}$.
In view of Lemma~B.1, $(u)_{i\in I}\in
\Fix(\bR_m\circ\cdots\circ\bR_2)$.
Hence $(\bR_m\circ\cdots\circ\bR_2\circ\bR_1)\bx = (u)_{i\in I}$ and
thus once again
$\|(\bR_m\circ\cdots\circ\bR_2\circ\bR_1)\bx \|^2 = m>m-1 =
\|\bx\|^2$.
This completes the proof.
\end{proof}

\noindent
\textbf{Proof of Remark~\ref{r:bT}.}
Let $\bx = (x_i)_{i\in I}$ and $\by = (y_i)_{i\in I}$
be in $\bX$, and take $k\in I$.
Using that $J_{\mu_k^{-1}A_k}$ is (firmly) nonexpansive
in \eqref{e:0201a},
and that $\|\cdot\|^2$ is convex in \eqref{e:0201b},
we obtain
\begin{subequations}
\begin{align}
&\qquad \big\|\big(\bJ_k\circ \bR_k\big)\bx - \big(\bJ_k\circ
\bR_k\big)\by\big\|^2\\[+3 mm]
&=
\bigg\|J_{\mu_k^{-1}A_k}\Big(\sum_{j\in I\smallsetminus \{k\}}
\frac{\lambda_j}{\mu_k^{}}x_j\Big)
- J_{\mu_k^{-1}A_k}\Big(\sum_{j\in I\smallsetminus \{k\}}
  \frac{\lambda_j}{\mu_k^{}}y_j\Big)\bigg\|^2 +
\sum_{j\in I\smallsetminus\{k\}} \|x_j-y_j\|^2\\
&\leq \label{e:0201a}
\bigg\|\Big(\sum_{j\in I\smallsetminus \{k\}}
\frac{\lambda_j}{\mu_k^{}}x_j\Big)
- \Big(\sum_{j\in I\smallsetminus \{k\}}
  \frac{\lambda_j}{\mu_k^{}}y_j\Big)\bigg\|^2 +
\sum_{j\in I\smallsetminus\{k\}} \|x_j-y_j\|^2\\
&=
\bigg\|\sum_{j\in I\smallsetminus \{k\}}
\frac{\lambda_j}{\mu_k^{}}\big(x_j-y_j\big) \bigg\|^2 +
\sum_{j\in I\smallsetminus\{k\}} \|x_j-y_j\|^2\\
&\leq \label{e:0201b}
\sum_{j\in I\smallsetminus \{k\}}
\frac{\lambda_j}{\mu_k^{}}\big\|x_j-y_j\big\|^2 +
\sum_{j\in I\smallsetminus\{k\}} \|x_j-y_j\|^2\\
&= \sum_{j\in I\smallsetminus \{k\}}
\frac{\lambda_j+\mu_k}{\mu_k^{}}\big\|x_j-y_j\big\|^2.
\end{align}
\end{subequations}
Since
$\big(\lambda_i\big)_{i\in I} = \big(\tfrac{1}{m}\big)_{i\in I}$,
we further deduce that
\begin{subequations}
\begin{align}
\big\|\big(\bJ_k\circ \bR_k\big)\bx - \big(\bJ_k\circ
\bR_k\big)\by\big\|^2
& \leq \sum_{j\in I\smallsetminus \{k\}}
\frac{m}{m-1}\big\|x_j-y_j\big\|^2\\
& \leq \sum_{j\in I}
\frac{m}{m-1}\big\|x_j-y_j\big\|^2\\
&=\frac{m}{m-1}\|\bx-\by\|^2,
\end{align}
\end{subequations}
which implies that $\bJ_k\circ \bR_k$ is Lipschitz continuous
with constant $\sqrt{m/(m-1)}$.
The rest of Remark~\ref{r:bT} now follows from elementary calculus.
\hfill $\quad \blacksquare$

\section*{Acknowledgments}
Heinz Bauschke was partially supported by the Natural Sciences and
Engineering Research Council of Canada and by the Canada Research Chair
Program.
Xianfu Wang was partially
supported by the Natural Sciences and Engineering Research Council
of Canada.
Calvin Wylie was partially supported
% by an Undergraduate Research Award sponsored
by the Irving K.\ Barber Endowment Fund.


\small

\begin{thebibliography}{999}

\bibitem{BBR}
J.B.\ Baillon, R.E.\ Bruck, and S.\ Reich,
On the asymptotic behavior of nonexpansive mappings
and semigroups in Banach spaces,
\emph{Houston Journal of Mathematics} 4 (1978), 1--9.

%\bibitem{BH}
%J.-B.\ Baillon and G.\ Haddad,
%Quelques propri\'et\'es des op\'erateurs angle-born\'es et
%$n$-cycliquement monotones,
%\emph{Israel Journal of Mathematics} 26 (1977), 137--150.

%\bibitem{BBBRW}
%S.\ Bartz, H.H.\ Bauschke, J.M.\ Borwein,
%S.\ Reich, and X.\ Wang,
%Fitzpatrick functions, cyclic monotonicity and Rockafellar's
%antiderivative,
%\emph{Nonlinear Analysis} 66 (2007), 1198--1223.

%\bibitem{BBC}
%H.H.\ Bauschke, J.M.\ Borwein, and P.L.\ Combettes,
%Essential smoothness, essential strict convexity,
%and Legendre functions in Banach spaces,
%\emph{Communications in Contemporary Mathematics} 3 (2001), 615--647.

\bibitem{BBL}
H.H.\ Bauschke, J.M.\ Borwein, and A.S.\ Lewis,
The method of cyclic projections for closed convex sets in Hilbert
space, in
\emph{Recent Developments in Optimization Theory and Nonlinear
Analysis (Jerusalem 1995)},
Y.\ Censor and S.\ Reich (editors),
Contemporary Mathematics~vol.~204,
American Mathematical Society, % Providence, R.I.,
pp.~1--38, 1997.

%\bibitem{BC2010}
%H.H.\ Bauschke and P.L.\ Combettes,
%The Baillon-Haddad theorem revisited,
%\emph{Journal of Convex Analysis} 17 (2010), 781--787.
%
\bibitem{BC2011}
H.H.\ Bauschke and P.L.\ Combettes,
\emph{Convex Analysis and Monotone Operator Theory in Hilbert Spaces},
Springer-Verlag, 2011.

\bibitem{BauEd}
H.H.\ Bauschke and M.R.\ Edwards,
A conjecture by De Pierro is true for translates of regular
subspaces,
\emph{Journal of Nonlinear and Convex Analysis}~6 (2005), 93--116.

\bibitem{Bjorck}
{\AA}.\ Bj{\"o}rck,
\emph{Numerical Methods for Least Squares Problems},
SIAM, 1996.

\bibitem{BorVanBook}
J.M.\ Borwein and J.D.\ Vanderwerff,
\emph{Convex Functions},
Cambridge University Press, 2010.


\bibitem{Brezis}
H. Br\'ezis,
\emph{Operateurs Maximaux Monotones et
Semi-Groupes de Contractions dans les Espaces de Hilbert},
North-Holland/Elsevier, 1973. % New York

\bibitem{Browder67}
F.E.\ Browder,
Convergence theorems for sequences of nonlinear operators
in Banach spaces,
\emph{Mathematische Zeitschrift}~100 (1967), 201--225.

\bibitem{BR}
R.E.\ Bruck and S.\ Reich,
Nonexpansive projections and resolvents of accretive operators
in Banach spaces,
\emph{Houston Journal of Mathematics}~3 (1977), 459--470.

\bibitem{BurIus}
R.S.\ Burachik and A.N.\ Iusem,
\emph{Set-Valued Mappings and Enlargements
of Monotone Operators},
Springer-Verlag, 2008.

\bibitem{Byrne05}
C.L.\ Byrne,
\emph{Signal Processing},
AK Peters, 2005.

\bibitem{Byrne08}
C.L.\ Byrne,
\emph{Applied Iterative Methods},
AK Peters, 2008.

\bibitem{CEG}
Y.\ Censor, P.P.B.\ Eggermont, and D.\ Gordon,
Strong underrelaxation in Kaczmarz's method for inconsistent
systems,
\emph{Numerische Mathematik}~41 (1983), 83--92.

%\bibitem{CIZ}
%Y. Censor, A.N.\ Iusem, and S.A.\ Zenios,
%An interior point method with Bregman functions
%for the variational inequality problem with
%paramonotone operators,
%\emph{Mathematical Programming Series~A}~81 (1998), 373--400.

\bibitem{Comb94}
P.L.\ Combettes,
Inconsistent signal feasibility problems:
least-squares solutions in a product space,
\emph{IEEE Transactions on Signal Processing}~42 (1994), 2955--2966.

\bibitem{Comb04}
P.L.\ Combettes,
Solving monotone inclusions via compositions of nonexpansive averaged
operators,
\emph{Optimization}~53 (2004), 475--504.

%\bibitem{Cross}
%R.\ Cross,
%\emph{Multivalued Linear Operators},
%Marcel Dekker, 1998.

\bibitem{DeP}
A.R.\ De Pierro,
From parallel to sequential projection methods and vice versa
in convex feasibility: results and conjectures, in
\emph{Inherently Parallel Algorithms in Feasibility and
Optimization and Their Applications (Haifa 2000)},
D.\ Butnariu, Y.\ Censor, and S.\ Reich (editors),
Elsevier, %Amsterdam, The Netherlands
pp.~187--201, 2001.

%\bibitem{Deutsch}
%F.\ Deutsch,
%\emph{Best Approximation in Inner Product Spaces},
%Springer-Verlag, 2001.

\bibitem{EckBer}
J.\ Eckstein and D.P.\ Bertsekas,
\emph{On the Douglas-Rachford splitting method
and the proximal point algorithm for maximal monotone
operators},
\emph{Mathematical Programming Series A} 55 (1992), 293--318.

%\bibitem{EKN}
%L.\ Elsner, I.\ Koltracht, and M.\ Neumann,
%Convergence of sequential and asynchronous
%nonlinear paracontractions,
%\emph{Numerische Mathematik}~62 (1992), 305--319.

\bibitem{GK}
K.\ Goebel and W.A.\ Kirk,
\emph{Topics in Metric Fixed Point Theory},
Cambridge University Press, 1990.

\bibitem{GR}
K.\ Goebel and S.\ Reich,
\emph{Uniform Convexity, Hyperbolic Geometry, and Nonexpansive Mappings},
Marcel Dekker, 1984.

\bibitem{Johnson}
R.A.\ Johnson,
\emph{Advanced Euclidean Geometry},
Dover Publications,  % New York
1960.

\bibitem{Krey}
E.\ Kreyszig,
\emph{Introductory Functional Analysis with Applications},
Wiley, 1989.

\bibitem{Meyer}
C.D.\ Meyer,
\emph{Matrix Analysis and Applied Linear Algebra},
SIAM, 2000.

\bibitem{Minty}
G.J.\ Minty,
Monotone (nonlinear) operators in Hilbert spaces,
\emph{Duke Mathematical Journal} 29 (1962), 341--346.


%\bibitem{Moreau}
%J.-J.\ Moreau,
%Proximit\'e et dualit\'e dans un espace hilbertien,
%\emph{Bulletin de la Soci\'et\'e Math\'ematique de France} 93 (1965),
%273--299.

\bibitem{Opial67}
Z.\ Opial,
Weak convergence of the sequence of successive approximations
for nonexpansive mappings,
\emph{Bulletin of the American Mathematical Society}~73 (1967), 591--597.

\bibitem{Rock70}
R.T.\ Rockafellar,
\emph{Convex Analysis},
Princeton University Press, Princeton, 1970.

%\bibitem{Rock76}
%R.T.\ Rockafellar,
%Monotone operators and the proximal point algorithm,
%\emph{SIAM Journal on Control and Optimization} 14
%(1976), 877--898.

\bibitem{Rock98}
R.T.\ Rockafellar and R. J-B\ Wets,
\emph{Variational Analysis},
Springer-Verlag, %New York,
corrected 3rd printing, 2009.

\bibitem{Saad}
Y.\ Saad,
\emph{Iterative Methods for Sparse Linear Systems},
2nd edition, SIAM, 2003.

\bibitem{Simons1}
S.\ Simons,
\emph{Minimax and Monotonicity},
Springer-Verlag,
1998.

\bibitem{Simons2}
S.\ Simons,
\emph{From Hahn-Banach to Monotonicity},
Springer-Verlag,
2008.

\bibitem{WangBau}
X.\ Wang and H.H.\ Bauschke,
Compositions and averages of two resolvents:
relative geometry of fixed point sets and
a partial answer to a question by C.~Byrne,
March 2010, submitted. See also
\texttt{http://arxiv.org/abs/1003.4793}

\bibitem{Zalinescu}{C.\ Z\u{a}linescu},
\emph{Convex Analysis in General Vector Spaces},
World Scientific Publishing, 2002.

\bibitem{Zara}
E.H.\ Zarantonello,
Projections on convex sets in Hilbert space and spectral theory I.
Projections on convex sets, in
\emph{Contributions to Nonlinear Functional Analysis},
E.H.\ Zarantonello (editor), pp.~237--341, Academic Press, 1971.


\bibitem{Zeidler2a}
E.\ Zeidler,
\emph{Nonlinear Functional Analysis and Its Applications II/A:
Linear Monotone Operators},
Springer-Verlag, 1990.

\bibitem{Zeidler2b}
E.\ Zeidler,
\emph{Nonlinear Functional Analysis and Its Applications II/B:
Nonlinear Monotone Operators},
Springer-Verlag, 1990.

\bibitem{Zeidler1}
E.\ Zeidler,
\emph{Nonlinear Functional Analysis and Its Applications I:
Fixed Point Theorems},
Springer-Verlag, 1993.

\end{thebibliography}
\end{document}